\UseRawInputEncoding
\documentclass[11pt]{article}
\usepackage{float}
\usepackage{color}
\usepackage{subfigure}
\usepackage{amsmath}
\usepackage{amsthm}
\usepackage{amsfonts}
\usepackage{amssymb,latexsym}
\usepackage[all]{xy}
\usepackage{graphicx}
\usepackage[mathscr]{eucal}
\usepackage{verbatim}
\usepackage{hyperref}

\theoremstyle{plain}

    \newtheorem{thm}{Theorem}[section]
       \newtheorem{pro}{Proposition}[section]
       \newtheorem{lem}{Lemma}[section]
       \newtheorem{cor}{Corollary}[section]
       \newtheorem{defn}{Definition}[section]
       \newtheorem{rem}{Remark}[section]

\numberwithin{equation}{section}

\usepackage{graphicx}

\begin{document}
\title{Topological dimensions of random attractors for  stochastic partial differential equations with delay}

\author{Wenjie Hu$^{1,2}$,  Tom\'{a}s
Caraballo$^{3}$\footnote{Corresponding author.  E-mail address: caraball@us.es (Tom\'{a}s
Caraballo).}
\\
\small  1. The MOE-LCSM, School of Mathematics and Statistics,  Hunan Normal University,\\
\small Changsha, Hunan 410081, China\\
\small  2. Journal House, Hunan Normal University, Changsha, Hunan 410081, China\\
\small 3 Dpto. Ecuaciones Diferenciales y An\'{a}lisis Num\'{e}rico, Facultad de Matem\'{a}ticas,\\
\small  Universidad de Sevilla, c/ Tarfia s/n, 41012-Sevilla, Spain
}

\date {}
\maketitle

\begin{abstract}
The aim of this paper is to obtain an estimation of  Hausdorff as well as fractal dimensions of random attractors for a class of stochastic partial differential equations with delay. The stochastic equation is first transformed into a delayed random partial differential equation by means of a random conjugation, which is then recast into  an auxiliary Hilbert space. For the obtained equation, it is firstly proved that it generates a random dynamical system (RDS)  in the auxiliary Hilbert space. Then it is shown that the equation possesses random attractors by a uniform estimate of the solution and  the asymptotic compactness of the generated RDS.  After establishing the variational equation in the auxiliary Hilbert space and the $\mathbb{P}$ almost surely differentiable properties of the RDS, an upper estimate of both Hausdorff and fractal dimensions of the random attractors are obtained.
\end{abstract}

\bigskip


\section{Introduction}
Existence and estimation of topological dimensions of  attractors play  important roles in the study of the long time behavior of deterministic or random dynamical systems. For many infinite dimensional systems generated by deterministic or stochastic partial differential equations and delay differential equations, the existence of attractors can  reduce the essential part of the flow to a compact set. The finite dimensionality of the attractors, which represents the number of degrees of freedom presented in the long term dynamics of the system can further simplify global dynamics of complex nonlinear systems and hence it is of great significance.

The theory  of attractors for deterministic infinite dimensional dynamical systems has been well established (see the monograph \cite{JH}). On the other hand, the study of random attractors for RDSs dates back to the pioneer works \cite{9,10,FS}, where H. Crauel, F. Flandoli, B. Schmalfu\ss, amongst others, generalized the concept of global attractors of infinite dimensional dissipative systems and  established the basic framework of random attractors for infinite dimensional RDSs. Since then,  the existence, dimension estimation and qualitative properties of random attractors for various stochastic nonlinear evolution equations or stochastic functional differential equations have been investigated by many  researchers. For example, for the stochastic reaction-diffusion equation without time delay, Caraballo et al. \cite{11}, Gao et al. \cite{13} and Li and Guo \cite{12} explored the existence of global attractors on bounded domains. In \cite{24}, \cite{15}  and \cite{14}, the authors obtained the existence of global attractors on unbounded domains. For the stochastic reaction diffusion equation with  delay, the existence of random attractors and   their structure have been studied in \cite{17,16,18,LG20,25} and the references therein.

Criteria for the  finite Hausdorff dimensionality of attractors for deterministic fluid dynamics models have been derived by  Douady and  Oesterle \cite{DO}, which was later generalized by  Constantin,  Foias and  Temam \cite{CFT} (see also   Temam \cite{TR}). Then, it was further extended to the stochastic case in \cite{CF} and \cite{SB}, where the RDS is first linearized and  the global Lyapunov exponents of the  linearized mapping is then examined. The main difficulty of this method lies in controlling the difference between the original nonlinear RDS and its linearization, since in the stochastic case,  the attractor is a random set which is not uniformly bounded. A. Debussche  showed that the random attractors of many random dynamical systems generated by dissipative evolution equations  have finite Hausdorff dimension by an ergodicity argument in \cite{DA97} and further gave a precise bound on
the dimension by combining the method of   linearization and Lyapunov exponents in \cite{DA98}. With respect to the fractal dimensionality of random sets,  Langa    proved the finite fractal dimensionality of the random attractor associated to a model from fluid dynamics in \cite{L}.  Langa and Robinson generalized the method in \cite{DA98} to the fractal dimension by requiring differentiability of RDS in \cite{LR}. Recently, the above established framework were generalized and adopted to various stochastic and random evolution equations. For instance, Fan  proved the existence of random attractor and obtained an upper bound of the Hausdorff and fractal dimension of the random attractor for a stochastic wave equations in \cite{FX08} by using the method in \cite{DA98}. In the recent work  \cite{ZM}, Zhou and Zhao proved the finiteness of fractal dimension of random attractor for stochastic damped wave equation with linear multiplicative white noise.

Despite the fact that the  finite Hausdorff and fractal dimensionality of attractors for  abstract RDSs and applications to stochastic partial differential equations (SPDEs) have been extensively and intensively studied, to our best knowledge, the estimation of dimensions of  SPDEs with delay, i.e., the stochastic partial functional differential equations (SPFDEs) have not been extensively studied.  There are only some early results on the existence and local stability of solutions \cite{TLT02,CL99,HZ} and recent results on the existence and qualitative properties of random attractors \cite{14,25,LG20,HZ20,HZT}. Indeed, even the  dimension estimation of attractors for delayed partial differential equations is scare even for the deterministic case since. The only work concern about dimensions of attractors for partial functional differential equations (PFDEs) we can find are \cite{SW91} and the very recent work \cite{QY}. In this paper, we make an attempt to estimate topological dimensions of random attractors for the stochastic delayed partial differential equation. Specifically, we consider the following SPFDE with additive noise
\begin{equation}\label{1}
 \displaystyle\frac{du(t)}{dt}=A u(t)-\mu u(t)-L u_{t} +f\left(u_{t}\right)+\sum_{j=1}^{m} g_{j}\frac{\mathrm{d} w_{j}(t)}{\mathrm{d}t}.
\end{equation}
Here, $u(t)\in \mathbb{X}$ and $\mathbb{X}$ is an arbitrary  Hilbert space with norm $\|\cdot\|_{\mathbb{X}}$ and inner product $(\cdot,\cdot)_{\mathbb{X}}$. $A$ is a linear elliptic differential operator, $u_{t}$ is an element of $\mathcal{L}$ defined by $u_{t}(\xi)=u(t+\xi)$ for $\xi \in [-\tau, 0]$, where $\mathcal{L}\triangleq L^2([-\tau,0], \mathbb{X})$ is the Hilbert space  of all square Lebesgue  integral functions from $[-\tau, 0]$ to $\mathbb{X}$ equipped with the  norm  $\|\varphi\|_{\mathcal{L}}=[\int_{-\tau}^0\|\varphi(s)\|_{\mathbb{X}}^2ds]^{1/2}$  for all $\varphi\in \mathcal{L}$. $L: \mathcal{L} \mapsto \mathbb{X}$ is a bounded linear operator,  $f : \mathcal{L} \mapsto \mathbb{X}$ is an everywhere defined and nonlinear Lipschitz continuous operator. $\{g_j\}_{j=1}^{m}\subseteq \mathbb{X}$ and  $\{Ag_j\}_{j=1}^{m}\subseteq \mathbb{X}$ stand for the intensity and the shape of noise, $\{w_j\}_{j=1}^{m}$ are mutually independent two-sided real-valued Wiener process on an appropriate probability space to be specified below. Equation \eqref{1} can model  many processes from chemistry or mathematical biology. For instance, in the case $A=\Delta$, \eqref{1} can describe the evolution of mature populations for age-structured species, where $A,\mu$ represent spatial diffusion and death rate of mature individuals, $Lu_t$ and $f(u_t)$ represent death rate of immature individuals  and  birth rate respectively, $\sum_{j=1}^{m} g_{j}\frac{\mathrm{d} w_{j}(t)}{\mathrm{d}t}$ stands for the random perturbations or environmental effects.

The main difficulty for studying the topological dimensions  of \eqref{1} lies in the fact that the natural phase spaces for deterministic or stochastic PFDEs  are Banach spaces while all the above mentioned theories are established for dynamical systems in Hilbert spaces. Hence,  in \cite{SW91},  the authors associated the deterministic PFDE  with a nonlinear semigroup on a product space, i.e. a Hilbert space. In this paper, we extend the method established in \cite{SW91} to the stochastic case. Nevertheless, the extension is not trivial since the RDSs are nonautonomous in nature and the random attractor  is not uniformly bounded. In \cite{SW91},  the authors assumed that the deterministic PFDEs are dissipative which directly implies the existence of attractors in the auxiliary Hilbert space.  In this paper, we will also give details for proving  the  existence of random attractors   for \eqref{1} in the auxiliary Hilbert space.

The rest of  this paper is organized as follows. In Section 2, we introduce some notation, hypotheses and  recast \eqref{1}  into a Hilbert space.  In Section 3, we prove  the obtained auxiliary equation admits a global mild solution which generates a RDS and possesses random attractors under certain conditions. In Section 4, we obtain an upper bound of the Hausdorff and fractal dimensions for the random attractors of the auxiliary equation, which directly implies the finite dimensionality of the original equation \eqref{1}. Finally, we conclude the paper and point out some potential research directions.

\section{Auxiliary equation}
In this paper, we consider the canonical probability space $(\Omega, \mathcal{F}, P)$  with
$$
\Omega=\left\{\omega=\left(\omega_{1}, \omega_{2}, \ldots, \omega_{m}\right) \in C\left(\mathbb{R} ; \mathbb{R}^{m}\right): \omega(0)=0\right\}
$$
and $\mathcal{F}$ being the Borel $\sigma$-algebra induced by the compact open topology of $\Omega,$ while $P$ being the corresponding Wiener measure on $(\Omega, \mathcal{F})$. Then, we identify $W(t)$ with $\omega(t)$, i.e.,
$$
W(t)\equiv \left(\omega_{1}(t), \omega_{2}(t), \ldots, \omega_{m}(t)\right)\quad \text { for } t \in \mathbb{R}.
$$
and the time shift by $$\theta_{t} \omega(\cdot)=\omega(\cdot+t)-\omega(t), t \in \mathbb{R}.$$

In the following, we follow the idea of \cite{DLS03} to transform  \eqref{1} into a pathwise deterministic equation. The same idea has been adopted by many authors when dealing with random attractors or invariant manifolds for various stochastic evolution equations, such as \cite{DLS04,HZ20,LG20,LS07}.
Consider the stochastic stationary solution of the one dimensional Ornstein-Uhlenbeck equation
\begin{equation}\label{2.1}
\mathrm{d} z_{j}+\mu z_{j} \mathrm{d} t=\mathrm{d} w_{j}(t), j=1, \ldots, m,
\end{equation}
which is given by
\begin{equation}\label{2.2}
z_{j}(t) \triangleq z_{j}\left(\theta_{t} \omega_{j}\right)=-\mu \int_{-\infty}^{0} e^{\lambda s}\left(\theta_{t} \omega_{j}\right)(s) \mathrm{d} s, \quad t \in \mathbb{R}.
\end{equation}
By Definition \ref{defn4} (in Section 3), one can see that the random variable $\left|z_{j}\left(\omega_{j}\right)\right|$ is tempered and $z_{j}\left(\theta_{t} \omega_{j}\right)$ is $P$-a.e. $\omega$ continuous. Therefore,  Proposition 4.3.3 in \cite{AL} implies that there exists a tempered function $0<r(\omega)<\infty$ such that
\begin{equation}\label{2.7}
\sum_{j=1}^{m}\left|z_{j}\left(\omega_{j}\right)\right|^{2} \leq r(\omega),
\end{equation}
where $r(\omega)$ satisfies, for $P$-a.e. $\omega \in \Omega$,
\begin{equation}\label{3.8}
r\left(\theta_{t} \omega\right) \leq e^{\frac{\mu}{2}|t|} r(\omega), \quad t \in \mathbb{R}.
\end{equation}
Combining \eqref{3.7} with \eqref{3.8}, we obtain that  for $P$-a.e. $\omega \in \Omega$,
\begin{equation}\label{3.9}
\sum_{j=1}^{m}\left|z_{j}\left(\theta_t\omega_{j}\right)\right|^{2} \leq e^{\frac{\mu}{2}|t|} r(\omega), \quad t \in \mathbb{R}.
\end{equation}
Moreover, we have
\begin{equation}\label{3.9a}
\sum_{j=1}^{m}\left|z_{j}\left(\theta_{\xi}\omega_{j}\right)\right|^{2} \leq  e^{\frac{\mu\tau}{2}}r(\omega),
\end{equation}
for any $\xi\in [-\tau,0]$ and $P$-a.e. $\omega \in \Omega$.
Putting $z\left(\theta_{t} \omega\right)=\sum_{j=1}^{m} g_{j} z_{j}\left(\theta_{t} \omega_{j}\right)$, we have
$$
\mathrm{d} z+\mu z \mathrm{d} t=\sum_{j=1}^{m} g_{j} \mathrm{d} w_{j}.
$$
 Take the transformation $v(t)=u(t)-z\left(\theta_{t} \omega\right)$. Then,  simple computation gives
\begin{equation}\label{2.3}
\displaystyle \frac{dv(t)}{dt}=\displaystyle  Av(t)-\mu v(t)-Lv_t-L z(\theta_{t+\cdot}\omega)+f\left(v_t+z(\theta_{t+\cdot}\omega)\right)+Az(\theta_{t}\omega).
\end{equation}

In order to estimate topological dimensions  of the random attractors of  \eqref{1}, unlike previous works \cite{17,16,18,LG20,25}, where $v_t$ is taken as the state and $\mathcal{L}$ as state space for the above obtained pathwise deterministic delayed equation \eqref{2.3}, we take $V(t)=(v_t,v(t))$ as state space and recast the equation into an auxiliary product space $H=\mathcal{L}\times \mathbb{X}$ equipped with the inner product
$$
((\phi, h),(\psi, k))=\int_{-\tau}^{0}(\phi(s), \psi(s))_{\mathbb{X}} ds+(h, k)_{\mathbb{X}} \quad \text { for } \quad(\phi, h), (\psi, k) \in H
$$
and norm
$$
\|(\phi, h)\|=((\phi, h),(\phi, h))^{1 / 2} \quad \text { for } \quad(\phi, h) \in H,
$$
making $H$ a Hilbert space and hence we can overcome the lack of Hilbert space geometry in applying  the abstract theory established in \cite{DA97,DA98,L,LR}. Furthermore,  recasting \eqref{1} into the Hilbert space $H$ also facilitate us to construct an appropriate variational equation. Take  $V(t)=(v_t, v(t))^T$,
\begin{equation}\label{2.4a}
\tilde{f}(t, \theta_{t}\omega,v_t)\triangleq Az(\theta_{t}\omega)-L z(\theta_{t+\cdot}\omega)+f\left(v_t+z(\theta_{t+\cdot}\omega)\right)
\end{equation}
and
\begin{equation}\label{2.4b}
F(t, \theta_{t}\omega,V(t))=(0, \tilde{f}(t,\theta_{t}\omega,v_t)).
\end{equation}
We consider the following auxiliary random partial differential equation on $H$.
\begin{equation}\label{2.4}
\left\{\begin{array}{l}
\displaystyle \frac{dV(t)}{dt}=\displaystyle  \tilde{A}V(t)- \tilde{L}V(t)+F(t, \theta_{t}\omega, V(t)),\\
V(0)=(\phi,h), \quad(\phi, h) \in H,
\end{array}\right.
\end{equation}
where  operator $\tilde{A}$ is defined as
\begin{equation}\label{Atilde}
\tilde{A}:=\left(\begin{array}{cc}
\frac{d}{d t} & 0 \\
0 & A
\end{array}\right),
\end{equation}
with domain
$$
\begin{aligned}
D(\tilde{A})=\{(\phi, h) \in H: \phi &  \text { is differentialble on }[-\tau, 0], \left.\dot{\phi} \in \mathcal{L} \text { and } h=\phi(0) \in D(A)\right\}.
\end{aligned}
$$
The linear operator $\tilde{L}$ is defined by
$$
\tilde{L}:=\left(\begin{array}{cc}
-L & 0 \\
0 & -\mu I
\end{array}\right).
$$

Throughout the remaining part of this paper, we always impose the following assumptions on $A, L$ and the nonlinear term $f$:

$\mathbf{Hypothesis \    A1}$ $A: D(A) \subset \mathbb{X} \rightarrow \mathbb{X}$ is a densely defined linear operator  that  generates a strongly continuous compact semigroup $S(t)$ on $\mathbb{X}$.  $L$ satisfies $\|L\|\triangleq \sup _{\phi \in \mathcal{L}, \|\phi\|_\mathcal{L}=1}\|L\phi\|_\mathbb{X}\leq \mu$. Moreover, $\varrho\triangleq s(\tilde{A})-\mu<0$, where $s(\tilde{A})$ is defined by  $s(\tilde{A}):=\sup \{\Re \lambda: \lambda \in \sigma(\tilde{A})\}$ representing the spectral bound of the linear operator $\tilde{A}$.

$\mathbf{Hypothesis \  A2}$   $f$ is Lipschitz continuous with $\mathbf{0}$ being a fixed point, that is, $f(\mathbf{0})=\mathbf{0}$ and $\|f(\phi)-f(\varphi)\|_{\mathbb{X}}\leq L_f\|\phi-\varphi\|_{\mathcal{L}}$ for any $\phi, \varphi\in \mathcal{L}$.

It follows from $\mathbf{Hypothesis \  A1}$ that
\begin{equation}\label{2.4e}
\begin{aligned}
\|\tilde{L}\|\triangleq \sup _{\varphi \in H, \|\varphi\|=1}\|\tilde{L}\varphi\|\leq \mu.
\end{aligned}
\end{equation}
It follows from $\mathbf{Hypothesis \    A1}$, Lemma 3.6, Theorem 3.25 in \cite{BP} that the operator $(\tilde{A},D(\tilde{A}))$ is closed and densely defined on $H$, and generates a strongly continuous semigroup $\tilde{S}(t)$ given by
$$
\tilde{S}(t):=\left(\begin{array}{cc}
S(t) & 0 \\
S_t & T_0(t)
\end{array}\right),
$$
where $\left(T_0(t)\right)_{t \geq 0}$ is the nilpotent left shift semigroup on $\mathcal{L}$, and $S_t: \mathbb{X} \rightarrow \mathcal{L}$ is defined by
$$
\left(S_t x\right)(\xi):= \begin{cases}S(t+\xi) x & \text { if }-t<\xi \leq 0, \\ 0 & \text { if }-\tau \leq \xi \leq-t .\end{cases}
$$
 Moreover, by Theorem 4.11 in \cite{BP},  we have
$$
\|\tilde{S}(t)\|\leqslant e^{s(\tilde{A}) t}, \quad t \geqslant 0.\\
$$

Since for $P$-a.e. $\omega \in \Omega$, \eqref{2.4} is a path-wise deterministic equation, it follows from Theorem 6.1.5 Pazy \cite{PA} that  \eqref{2.4} admits a global classical solution which can be represented by a integral equation based on the variation of  constants formula.

\begin{thm} \label{thm2.1}
Assume that $\mathbf{Hypothesis \    A1}$ holds and $f$ is continuously differentiable. Then, for each $(\phi, h) \in H$, there exists a continuous function $V(\cdot,\omega, (\phi, h)):[0, \infty) \rightarrow H$ such that
\begin{equation}\label{2.5}
V(t,\omega, (\phi, h))=e^{-\tilde{L}t}\tilde{S}(t)(\phi, h)+\int_{0}^{t} e^{-\tilde{L}(t-s)}\tilde{S}(t-s)F(s, \theta_{s}\omega, V(s,\omega, (\phi, h))) d s, \quad t \geqslant 0
\end{equation}
for $P$-a.e. $\omega \in \Omega$.
Moreover, if $(\phi, h) \in D(\tilde{A})$, then $V(t,\omega, (\phi, h))$ is a strong solution of  \eqref{2.4}.
\end{thm}
\begin{rem}\label{rem2.1}
For the purpose of proving existence of random attractors and estimating their dimensions,  we always need the existence of strong solution to \eqref{2.5}. Therefore, in the remaining part of this paper, if not particularly specified, we always assume that initial condition $(\phi, h) \in D(\tilde{A})$ and hence $h=\phi(0)$.
\end{rem}
\section{Random attractors}
This section is devoted to showing  the existence of random attractors for the auxiliary equation  \eqref{2.4}. In the sequel, we first introduce the concept of random attractor and random dynamical systems following \cite{AL} and  \cite{9,10,FS}. Subsequently, we prove the existence of tempered pullback attractors for the the auxiliary equation  \eqref{2.4} by first establishing a uniform estimation for the solution and then proving that the RDS generated by \eqref{1} is  pullback asymptotically compact. Unlike the previous works \cite{17,16,18,25}, we prove the uniform a priori estimates of the solution by using the semigroup approach  instead of taking inner product.

\begin{defn}\label{defn1}
Let $\left\{\theta_{t}: \Omega \rightarrow \Omega, t \in \mathbb{R}\right\}$ be a family of measure preserving transformations such that $(t, \omega) \mapsto \theta_{t} \omega$ is measurable and $\theta_{0}=\mathrm{id}$, $\theta_{t+s}=\theta_{t} \theta_{s},$ for all $s, t \in \mathbb{R}$. The flow $\theta_{t}$ together with the probability space $\left(\Omega, \mathcal{F}, P,\left(\theta_{t}\right)_{t \in \mathbb{R}}\right)$ is called a metric dynamical system.
\end{defn}
 It follows from Definition \ref{defn1} that  $\left(\Omega, \mathcal{F}, P,\left(\theta_{t}\right)_{t \in \mathbb{R}}\right)$ is a metric dynamical system, where $\left(\Omega, \mathcal{F}, P\right)$ is defined in Section 2. Moreover, $\theta$ is ergodic. For a given separable Hilbert space $(H, \|\cdot\|_H)$,  denote by $\mathcal{B}(H)$ the  Borel-algebra of open subsets in $H$.
\begin{defn}\label{defn2}
A mapping $\Phi: \mathbb{R}^{+} \times \Omega \times H   \rightarrow H$ is said to be a random dynamical system (RDS) on a complete separable metric space  $(H,d)$ with Borel  $\sigma$-algebra  $\mathcal{B}(H)$ over the metric dynamical system $\left(\Omega, \mathcal{F}, P,\left(\theta_{t}\right)_{t \in \mathbb{R}^{+}}\right)$ if \\
(i) $\Phi(\cdot, \cdot, \cdot): \mathbb{R}^{+} \times \Omega \times H   \rightarrow H$ is $(\mathcal{B}(\mathbb{R}^{+})\times \mathcal{F}\times\mathcal{B}(H), \mathcal{B}(H))$-measurable;\\
(ii) $\Phi(0, \omega,\cdot)$ is the identity on  $H$ for $P$-a.e. $\omega \in \Omega$;\\
(iii) $\Phi(t+s, \omega,\cdot)=\Phi(t, \theta_{s} \omega,\cdot) \circ \Phi(s, \omega,\cdot),   \text { for all } t, s \in \mathbb{R}^{+}$ for $P$-a.e. $\omega \in \Omega$.\\
A RDS $\Phi$ is continuous or differentiable if $\Phi(t, \omega,\cdot): H \rightarrow H$ is continuous or differentiable for all $t\in \mathbb{R}^+$ and $P$-a.e. $\omega \in \Omega$.
\end{defn}
\begin{defn}\label{defn3}A set-valued map $\Omega \ni \omega \mapsto D(\omega) \in 2^{H}$, such that $D(\omega)$ is closed, is said to be a random set in $H$ if the mapping $\omega \mapsto d(x, D(\omega))$ is $(\mathcal{F}, \mathcal{B}(\mathbb{R}))$-measurable for any $x \in H,$ where $d(x, D(\omega))\triangleq \inf _{y\in  D(\omega)} \mathrm{d}(x, y)$ is the distance in $H$ between the element $x$ and the set $D(\omega)  \subset H$.
\end{defn}
\begin{defn}\label{defn4}A random set $\{D(\omega)\}_{\omega \in \Omega}$ of $H$ is called tempered with respect to $\left(\theta_{t}\right)_{t \in \mathbb{R}}$ if for $P$-a.e. $\omega \in \Omega$,
$$
\lim _{t \rightarrow \infty} e^{-\beta t} d\left(D\left(\theta_{-t} \omega\right)\right)=0, \quad \text { for all } \beta>0,
$$
where $d(D)=\sup _{x \in D}\|x\|_{H}$.
\end{defn}

\begin{defn}\label{defn5}
Let $\mathcal{D}=\{D(\omega)\subset H, \omega\in\Omega\}$ be a family of random sets.  A random set $K(\omega) \in \mathcal{D}$ is said to be a $\mathcal{D}$-pullback absorbing set for $\Phi$ if for $P$-a.e. $\omega \in \Omega$ and for every $B \in \mathcal{D},$ there exists $T=T(B, \omega)>0$ such that
\[
\Phi\left(t, \theta_{-t} \omega,B\left(\theta_{-t} \omega\right)\right)  \subseteq K(\omega) \quad \quad \text { for all } t \geq T.
\]
If, in addition, for all  $\omega \in \Omega, K(\omega)$ is measurable in $\Omega$ with respect to $\mathcal{F},$ then we say $K$ is a closed measurable $\mathcal{D}$-pullback absorbing set for $\Phi$.
\end{defn}

\begin{defn}\label{defn6}
 A RDS $\Phi$ is said to be $\mathcal{D}$-pullback asymptotically compact in $H$ if for $P$-a.e. $\omega \in \Omega$, $\left\{\Phi\left(t_{n}, \theta_{-t_{n}} \omega, x_{n}\right)\right\}_{n \geq 1}$ has a convergent subsequence in $H$ whenever $t_{n} \rightarrow \infty$ and $x_{n} \in D\left(\theta_{-t_{n}} \omega\right)$ for any given $D \in \mathcal{D}$.
\end{defn}

\begin{defn}\label{defn7}
A compact random set $\mathcal{A}(\omega)$ is said to be a $\mathcal{D}$-pullback random attractor associated to the RDS
$\Phi$ if it satisfies the invariance property
$$\Phi(t, \omega) \mathcal{A}(\omega)=\mathcal{A}\left(\theta_{t} \omega\right), \quad \text { for all } t \geq 0 $$
and the pullback attracting property
\[
\lim _{t \rightarrow \infty} \operatorname{dist}\left(\Phi\left(t, \theta_{-t} \omega\right)D\left(\theta_{-t} \omega\right), \mathcal{A}(\omega)\right)=0, \quad \text { for all } t \geq 0, D \in \mathcal{D}, P-a.e.\  \omega\in \Omega.
\]
where  $\operatorname{dist} (\cdot, \cdot)$ denotes the Hausdorff semidistance
\[
\operatorname{dist}(A, B)=\sup _{x \in A} \inf _{y\in  B} \mathrm{d}(x, y), \quad A, B \subset  H.
\]
\end{defn}

\begin{lem}\label{lem1}
Let $(\theta, \Phi)$ be a continuous random dynamical system. Suppose that $\Phi$ is $\mathcal{D}$-pullback asymptotically compact and has a closed pullback $\mathcal{D}$-absorbing random set $K=\{K(\omega)\}_{\omega \in \Omega} \in \mathcal{D}$. Then it possesses a random attractor $\{\mathcal{A}(\omega)\}_{\omega \in \Omega},$ where
$$
\mathcal{A}(\omega)=\cap_{\tau \geq 0} \overline{\cup_{t \geq \tau} \Phi\left(t, \theta_{-t} \omega, K\left(\theta_{-t} \omega\right)\right)}.
$$
\end{lem}

For convenience, we introduce the following Gr{o}nwall inequality in \cite{17} that will be frequently used in our subsequent proofs.
\begin{lem}\label{lem2}
Let $T>0$ and $u, \alpha, f$ and $g$ be non-negative continuous functions defined on $[0, T]$ such that
\[
u(t) \leq \alpha(t)+f(t) \int_{0}^{t} g(r) u(r) d r, \quad \text { for } t \in[0, T].
\]
Then,
\[
u(t) \leq \alpha(t)+f(t) \int_{0}^{t} g(r) \alpha(r) e^{\int_{r}^{t} f(\tau) g(\tau) d \tau} d r, \quad \text { for } t \in[0, T].
\]
\end{lem}

Apparently, under the conjugation transformation induced by \eqref{2.2}, no exceptional sets appear  in the equation \eqref{2.4}. By the uniqueness of solution to \eqref{2.4} for each $\omega \in \Omega$, we can see the  mapping
$\Phi(\cdot, \cdot, \cdot): \mathbb{R}^{+} \times \Omega \times H   \rightarrow H$  defined by
 \begin{equation}\label{3.20}
\Phi(t, \omega, (\phi,\phi(0)))=V(t, \omega, (\phi,\phi(0)))
\end{equation}
generates a RDS, which is $(\mathcal{B}(\mathbb{R}^{+})\times \mathcal{F}\times\mathcal{B}(H), \mathcal{B}(H))$-measurable. Let $P_{1}$ and $P_{2}$ be the projections of $H$ onto $\mathcal{L}$ and $\mathbb{X}$ respectively. Then, by Theorem 3.1 and Proposition 3.2 in  \cite{SW91}, we have
 \begin{equation}\label{2.5a}
v_t(\cdot, \omega, \phi)=P_{1} V(t,\omega, (\phi, \phi(0)))
\end{equation}
and
\begin{equation}\label{2.5b}
v(t, \omega, \phi)=P_{2} V(t,\omega, (\phi, \phi(0)))
\end{equation}
for $t \geqslant 0$ and P-a.e. $\omega\in \Omega$, where $v(t, \omega, \phi)$ is the solution to \eqref{2.3}. Therefore, the solution of \eqref{1} can be represented by
\begin{equation}\label{2.5c}
u_t(\cdot,\omega,\phi)=v_t(\cdot, \omega, \phi)+z(\theta_{t+\cdot}\omega)=P_{1} [V(t,\omega, (\phi, \phi(0)))+(z(\theta_{t+\cdot}\omega), z(\theta_{t}\omega))]\triangleq P_{1}\Psi(t,\omega, (\psi, \psi(0)))
\end{equation}
where
the mapping $\Psi: \mathbb{R}^{+} \times \Omega \times H \rightarrow H$ is defined by
\begin{equation}\label{2.5d}
\Psi(t, \omega,(\psi, \psi(0)))\triangleq\Phi(t, \omega, (\phi,\phi(0)))+(z(\theta_{t+\cdot}\omega), z(\theta_{t}\omega))=V(t,\omega, (\phi,\phi(0)))+(z(\theta_{t+\cdot}\omega), z(\theta_{t}\omega))
\end{equation}
and $(\psi, \psi(0))=(\phi,\phi(0)))+(z(\theta_{\cdot}\omega), z(\omega))$. By  the cocycle property of $z$ and $\Phi$, we can see that $\Psi$ is a  RDS on $H$. In the following, we show the existence of random attractor for $\Psi$.

\begin{lem}\label{lem3.3}
Assume that $\mathbf{Hypotheses \  A1-A2}$ hold and $\varrho\triangleq \max \{0,1 / 2+\sigma\}-\mu<\frac{-\mu}{2}$, $\varrho+L_f< 0$, then there exists $\{K(\omega)\}_{\omega \in \Omega} \in \mathcal{D}$ satisfying that, for any $B=\{B(\omega)\}_{\omega \in \Omega} \in \mathcal{D}$ and $P$-a.e. $\omega \in \Omega$, there is $T_{B}(\omega)>0$ such that
$$
\Psi \left(t, \theta_{-t} \omega, B\left(\theta_{-t} \omega\right)\right) \subseteq K(\omega) \quad \text { for all } t \geqslant T_{B}(\omega),
$$
that is, $\{K(\omega)\}_{\omega \in \Omega}$ is a random absorbing set for $\Psi$ in $\mathcal{D}$.
\end{lem}
\begin{proof}
We first derive uniform estimates of $V$ by \eqref{2.5} and then obtain the existence of an absorbing set for $\Psi$ given by $\Psi(t, \omega,(\phi,\phi(0))=V(t, \omega, (\phi,\phi(0)))+(z(\theta_{t+\cdot}\omega), z(\theta_{t}\omega))$. It follows from \eqref{2.5} that, for any $t>0$,
 \begin{equation}\label{4.1x}
\begin{aligned}
\|V(t,\omega,(\phi, \phi(0)))\|&=\|e^{-\tilde{L}t}\tilde{S}(t)(\phi, \phi(0))+\int_{0}^{t} e^{-\tilde{L}(t-s)}\tilde{S}(t-s)F(s, \theta_{s}\omega, V(s,\omega, (\phi, \phi(0)))) d s\|\\
&\leq e^{\varrho t}\|(\phi, \phi(0))\|+\int_{0}^{t} e^{\varrho(t-s)}\|\tilde{f}(s, \theta_{s}\omega, v_s(\cdot,\omega, \phi)))\|_\mathbb{X} d s\\
&\leq e^{\varrho t}\|(\phi, \phi(0))\|+\int_{0}^{t} e^{\varrho(t-s)}(\|Az(\theta_{s}\omega)\|_\mathbb{X}+(\mu+L_f)\|z(\theta_{s+\cdot}\omega)\|_\mathcal{L})d s\\
&+L_f\int_{0}^{t} e^{\varrho(t-s)}\|v_s(\cdot,\omega,\phi)\|_\mathcal{L}d s\\
&\leq e^{\varrho t}\|(\phi, \phi(0))\|+\int_{0}^{t} e^{\varrho(t-s)}(\|Az(\theta_{s}\omega)\|_\mathbb{X}+(\mu+L_f)\|z(\theta_{s+\cdot}\omega)\|_\mathcal{L})d s\\
&+L_f\int_{0}^{t} e^{\varrho(t-s)}\|V(s,\omega,(\phi, \phi(0)))\|d s\\
\end{aligned}
 \end{equation}
for  $P$-a.e. $\omega \in \Omega$. For the sake of simplicity, we denote $\varpi(\omega)=(\phi(\cdot,\omega), \phi(0,\omega))$. By replacing $\omega$ by $\theta_{-t} \omega$, we derive  from \eqref{4.1x} that, for all $t \geq 0,$
 \begin{equation}\label{4.2a}
\begin{aligned}
\|V(t,\theta_{-t} \omega,\varpi(\theta_{-t}\omega))\|
&\leq e^{\varrho t}\|\varpi(\theta_{-t}\omega)\|+L_f\int_{0}^{t} e^{\varrho(t-s)}\|V(s,\theta_{-t}\omega,\varpi(\theta_{-t}\omega)\|d s\\
&+\int_{0}^{t} e^{\varrho(t-s)}(\|Az(\theta_{-t} \theta_{s}\omega)\|_\mathbb{X}+(\mu+L_f)\|z(\theta_{-t} \theta_{s+\cdot}\omega)\|_\mathcal{L})d s\\
\end{aligned}
 \end{equation}
Since $g_j \in \mathbb{X}$, $Ag_j \in \mathbb{X}$  and $z\left( \omega\right)=\sum_{j=1}^{m} g_{j} z_{j}\left( \omega_{j}\right)$, it follows from \eqref{3.9} and \eqref{3.9a} that there exists a constant $c$ such that $p_1(\omega)\triangleq\|A z\left(\omega\right)\|_\mathbb{X}+(\mu+L_f ) \|z\left(\theta_{\cdot}\omega\right)\|_\mathcal{L}\leq c \sum_{j=1}^{m}\left|z_{j}\left(\omega_{j}\right)\right|^{2}$. Therefore, it follows from \eqref{3.8} and \eqref{3.9} that
\begin{equation}\label{4.5a}
\begin{aligned}
\int_{0}^{t} e^{\varrho(t-s)} p_{1}\left(\theta_{s-t} \omega\right) \mathrm{d} s \leq c \int_{0}^{t} e^{(\varrho+\frac{\mu}{2})(t-s)} r(\omega) \mathrm{d} s \leq c r(\omega),
\end{aligned}
\end{equation}
where the second inequality follows from the assumption that $\rho+\frac{\mu}{2}<0$.
Incorporating \eqref{4.5a} into \eqref{4.2a}  gives rise to
 \begin{equation}\label{4.6e}
\begin{aligned}
\|V(t,\theta_{-t} \omega,\varpi(\theta_{-t}\omega)\|
&\leq e^{\varrho t}\|\varpi(\theta_{-t}\omega)\|+L_f\int_{0}^{t} e^{\varrho(t-s)}\|V(s,\theta_{-t}\omega,\varpi(\theta_{-t}\omega))\|d s+c r(\omega).\\
\end{aligned}
 \end{equation}
Multiplying both sides of \eqref{4.6e} by $e^{-\varrho t}$,
 \begin{equation}\label{4.6f}
\begin{aligned}
e^{-\varrho t}\|V(t,\theta_{-t} \omega,\varpi(\theta_{-t}\omega))\|
&\leq \|\varpi(\theta_{-t}\omega)\|+L_f\int_{0}^{t} e^{-\varrho s}\|V(s,\theta_{-t}\omega,\varpi(\theta_{-t}\omega))\|d s+c e^{-\varrho t}r(\omega).\\
\end{aligned}
 \end{equation}
Hence, by the Gr\"{o}nwall inequality (Lemma \ref{lem2}), we have
 \begin{equation}\label{3.7}
\begin{aligned}
e^{-\varrho t}\|V(t,\theta_{-t} \omega,\varpi(\theta_{-t}\omega))\|
&\leq \|\varpi(\theta_{-t}\omega)\|+c e^{-\varrho t}r(\omega)+L_f\int_{0}^{t} e^{L_{f}(t-s)}(\|\varpi(\theta_{-s}\omega)\|+c e^{-\varrho s}r(\omega))d s\\
\leq & \|\varpi(\theta_{-t}\omega)\|+c e^{-\varrho t}r(\omega)+L_{f}\|\varpi(\theta_{-t}\omega)\|\int_{0}^{t}e^{ L_{f}(t-s)}\mathrm{d} s\\
&+cL_{f}r(\omega)\int_{0}^{t}e^{ L_{f}(t-s)}e^{-\varrho s}\mathrm{d} s.
\end{aligned}
 \end{equation}
Therefore, we have
 \begin{equation}\label{4.8c}
\begin{aligned}
\|V(t,\theta_{-t} \omega,\varpi(\theta_{-t}\omega))\|
&\leq e^{\varrho t}\|\varpi(\theta_{-t}\omega)\|+c r(\omega)+e^{\varrho t}(e^{L_{f}t}-1)\|\varpi(\theta_{-t}\omega)\|\\  &+\frac{c L_{f}}{\varrho+L_{f}}[e^{(L_{f}+\varrho) t}-1]r(\omega).
\end{aligned}
 \end{equation}
Note that $\Psi(t, \omega,\chi(\omega))=V(t,\omega, \varpi(\omega))+(z(\theta_{t+\cdot}\omega), z(\theta_{t}\omega))$
and $\chi(\omega)=\varpi(\omega)+(z(\theta_{\cdot}\omega), z(\omega))$. The above estimate \eqref{4.8c}  implies that, for all $t \geq 0$
 \begin{equation}\label{4.11e}
\begin{aligned}
\|\Psi(t,\theta_{-t} \omega,\chi(\theta_{-t}\omega))\|\leq & \left\|V\left(t, \theta_{-t} \omega, \varpi\left(\theta_{-t} \omega\right)\right)\right\|+\|(z(\theta_{-t}\theta_{t+\cdot}\omega), z(\theta_{-t}\theta_{t}\omega))\| \\
&\leq e^{\varrho t}\|\varpi(\theta_{-t}\omega)\|+2c r(\omega)+e^{\varrho t}(e^{L_{f}t}-1)\|\varpi(\theta_{-t}\omega)\|\\  &+\frac{c L_{f}}{\varrho+L_{f}}[e^{(L_{f}+\varrho) t}-1]r(\omega).
\end{aligned}
 \end{equation}

Therefore, if $\chi \in \mathcal{D}\left(\theta_{-t} \omega\right)$ and $L_{f}+\varrho<0$, then there exists a $T_{\mathcal{D}}>0$ such that, for all $t \geq T_{D}(\omega)$,
\begin{equation}\label{3.12}
\begin{aligned}
e^{\varrho t}\|\varpi(\theta_{-t}\omega)\|+e^{\varrho t}(e^{L_{f}t}-1)\|\varpi(\theta_{-t}\omega)\|+\frac{c L_{f}}{\varrho+L_{f}}e^{(L_{f}+\varrho) t}r(\omega)\leq c_1(\omega),
\end{aligned}
\end{equation}
which, along with \eqref{4.11e} shows that, for all $t \geq T_{\mathcal{D}}(\omega)$
\begin{equation}\label{4.13f}
\begin{aligned}
\|\Psi(t,\theta_{-t} \omega,\chi(\theta_{-t}\omega))\| \leq  2cr(\omega)+\frac{-c L_{f}}{\varrho+L_{f}}r(\omega)+c_1(\omega).
\end{aligned}
\end{equation}
Given $\omega \in \Omega,$ define
\begin{equation}\label{3.14}
\begin{aligned}
K(\omega)=\{\varphi \in H: \|\varphi\| \leq 2cr(\omega)+\frac{-c L_{f}}{\varrho+L_{f}}r(\omega)+c_1(\omega)\}.
\end{aligned}
\end{equation}
Then, $K=\{K(\omega)\}_{\omega \in \Omega} \in \mathcal{D}$. Furthermore, \eqref{4.13f} implies that $K(\omega)$ is a random absorbing set for the RDS $\Psi$ in $\mathcal{D}$.
\end{proof}

\begin{lem}\label{lem3.4}
Assume that $\mathbf{Hypothesis \  A1-A2}$ are satisfied and $\varrho\triangleq \max \{0,1 / 2+\sigma\}-\mu<\frac{-\mu}{2}$, $\varrho+L_f< 0$. Then, the RDS $\Psi$  is $\mathcal{D}$-pullback asymptotically compact  for $t>\tau$, i.e., for $P$-a.e. $\omega \in \Omega$, the sequence $\{\Psi(t_{n}, \theta_{-t_{n}} \omega,\phi_n\left(\theta_{-t_{n}} \omega\right))\}$ has a convergent subsequence provided $t_{n} \rightarrow \infty$, $B=\{B(\omega)\}_{\omega \in \Omega} \in \mathcal{D}$ and $\phi_n\left(\theta_{-t_{n}} \omega\right) \in B\left(\theta_{-t_{n}} \omega\right)$.
\end{lem}
\begin{proof}
Take an arbitrary random set $\{B(\omega)\}_{\omega \in \Omega} \in \mathcal{D}$, a sequence $t_{n} \rightarrow+\infty$ and $\phi_n \in B\left(\theta_{-t_{n}} \omega\right)$. We have to prove that $\left\{\Psi\left(t_{n}, \theta_{-t_{n}} \omega, \phi_n\right)\right\}$ is precompact. Since $\{K(\omega)\}$ is a random absorbing set for $\Psi$, there exists $T>0$ such that, for all $\omega \in \Omega$,
\begin{equation}\label{4.24}
\Psi\left(t, \theta_{-t} \omega\right) B\left(\theta_{-t} \omega\right) \subset K(\omega)
\end{equation}
for all $t \geq T$.
Because $t_{n} \rightarrow+\infty$, we can choose $n_{1} \geq 1$ such that $t_{n_{1}}-1 \geq T$. Applying \eqref{4.24} for $t=t_{n_{1}}-1$ and $\omega=\theta_{-1} \omega$, we find that
\begin{equation}\label{4.25}
\eta_{1} \triangleq \Psi\left(t_{n_{1}}-1, \theta_{-t_{n_{1}}} \omega, \phi_{n_{1}}\right)  \in K\left(\theta_{-1} \omega\right)
\end{equation}
Similarly, we can choose a subsequence $\left\{n_{k}\right\}$ of $\{n\}$ such that $n_{1}<n_{2}<\cdots<n_{k} \rightarrow$ $+\infty$ such that $t_{n_k}\geq k$ and
\begin{equation}\label{4.26}
\eta_{k} \triangleq \Psi\left(t_{n_{k}}-k, \theta_{-t_{n_{k}}} \omega, \phi_{n_{k}}\right)  \in K\left(\theta_{-k} \omega\right)
\end{equation}
Hence, by the assumptions we conclude that
the sequence
\begin{equation}\label{4.27}
\left\{\Psi\left(k, \theta_{-k} \omega, \eta_{k}\right)\right\}\ \mbox {is precompact.}
\end{equation}
On the other hand, by \eqref{4.26} we have
\begin{equation}\label{4.28}
\begin{aligned}
\Psi(k, \theta_{-k} \omega, \eta_{k}) &=\Psi(k, \theta_{-k} \omega,\Psi(t_{n_{k}}-k, \theta_{-t_{n_{k}}} \omega, \phi_{n_{k}}) =\Psi\left(t_{n_{k}}, \theta_{-t_{n_{k}}} \omega,\phi_{n_{k}}\right),
\end{aligned}
\end{equation}
for all $k \geq 1$. Combining \eqref{4.27} and \eqref{4.28}, we obtain that the sequence $\left\{\Psi\left(t_{n_{k}}, \theta_{-t_{n_{k}}} \omega,\phi_{n_{k}}\right)\right\}$ is precompact. Therefore,  $\left\{\Psi\left(t_{n}, \theta_{t_{n}} \omega,\phi_{n}\right) \right\}$ is precompact, which completes the proof.
\end{proof}

Lemma \ref{lem3.3} says that the continuous RDS $\Psi$ has a random absorbing set while Lemma \ref{lem3.4} tells us that $(\theta, \Psi)$ is pullback asymptotically compact in $H$. Thus, it follows from Lemma \ref{lem1} that the continuous RDS $(\theta, \Psi)$ possesses a random attractor. Namely, we obtain the following result.

\begin{thm}\label{thm3.1} Assume that $\mathbf{Hypotheses \  A1-A2}$ are satisfied and $\varrho\triangleq \max \{0,1 / 2+\sigma\}-\mu<\frac{-\mu}{2}$, $\varrho+L_f< 0$, then the continuous RDS $\Psi$  admits a unique $\mathcal{D}$-pullback attractor $\mathcal{A}_\Psi$ in $H$ belonging to the class $\mathcal{D}$.
\end{thm}

Moreover, by Theorem 3.1, the relationship between the RSDs $\Phi$ and $\Psi$ defined by \eqref{2.5d} as well as Proposition 3.2 in  \cite{SW91},  we have the following result about the existence of random attractors for equations  \eqref{1} and \eqref{2.4}.
\begin{cor}\label{cor3.1} Assume that $\mathbf{Hypotheses \  A1-A2}$ are satisfied and $\varrho\triangleq \max \{0,1 / 2+\sigma\}-\mu<\frac{-\mu}{2}$, $\varrho+L_f< 0$. Then, the continuous RDS $P_1\Psi$ generated by \eqref{1} admits a unique pullback attractor $P_1\mathcal{A}_\Psi$ in $P_1H$. Moreover, $\mathcal{A}_\Phi\triangleq \{\zeta|\zeta=\chi-(z(\theta_{t+\cdot}\omega), z(\theta_{t}\omega)), \chi\in \mathcal{A}_\Psi\}$ is the random attractor of $\Phi$.
\end{cor}
\section{Topological dimensions of random attractors}
The aim of this section is to estimate the Hausdorff and fractal dimensions for the attractor of \eqref{1}. Denote by $d_H(\mathcal{A})$ and $d_F(\mathcal{A})$ the Hausdorff and fractal dimensions of a random set $\mathcal{A}$ respectively.  We only need to prove that there exist  constants $d_H$ and $d_F$ such that $d_H(\mathcal{A}_\Psi)\leq d_H$ and $d_F(\mathcal{A}_\Psi)\leq d_F$, since by Theorem 3.1 and Proposition 3.2 in  \cite{SW91}, the topological dimensions of attractor $P_{1}\mathcal{A}_\Psi$ for \eqref{1}  apparently satisfy $d_H(P_{1}\mathcal{A}_\Psi)\leq d_H$ and $d_F(P_{1}\mathcal{A}_\Psi)\leq d_F$, i.e., the random attractors of \eqref{1} have finite Hausdorff and fractal dimensions less than those of \eqref{2.4}. In the sequel, we investigate the Hausdorff and fractal dimensions for the random attractor $\mathcal{A}_\Psi$ of \eqref{2.4}.

We first recall the concepts of Hausdorff and fractal dimensions of the attractor  $\mathcal{A}_\Psi\subset H$. More details can be found in \cite{DA98} and \cite{LR}. The Hausdorff dimension of the compact set $\mathcal{A}_\Psi\subset H$ is
$$
d_{H}(\mathcal{A}_\Psi)=\inf \left\{d: \mu_{H}(\mathcal{A}_\Psi, d)= 0 \right\}
$$
where, for $d \geq 0$,
$$\mu_{H}(\mathcal{A}_\Psi, d)=\lim _{\varepsilon \rightarrow 0} \mu_{H}(\mathcal{A}_\Psi, d, \varepsilon)$$
 denote the $d$-dimensional Hausdorff measure of the set $\mathcal{A}_\Psi\subset H$, where
 $$\mu_{H}(\mathcal{A}_\Psi, d, \varepsilon)=\inf \sum_{i} r_{i}^{d}$$
 and the infimum is taken over all coverings of $\mathcal{A}_\Psi$ by balls of radius $r_{i} \leqslant \varepsilon$. It can be shown that there exists $d_{H}(\mathcal{A}_\Psi) \in[0,+\infty]$ such that $\mu_{H}(\mathcal{A}_\Psi, d)=0$ for $d>d_{H}(\mathcal{A}_\Psi)$ and $\mu_{H}(\mathcal{A}_\Psi, d)=\infty$ for $d<d_{H}(\mathcal{A}_\Psi)$. $d_{H}(\mathcal{A}_\Psi)$ is called the Hausdroff dimension of $\mathcal{A}_\Psi$.

The fractal dimension (or capacity) of $\mathcal{A}_\Psi$ is defined as
$$
d_{F}(\mathcal{A}_\Psi)=\inf \left\{d>0: \mu_{F}(\mathcal{A}_\Psi, d)=0\right\},
$$
where
$$\mu_{F}(\mathcal{A}_\Psi, d)=\limsup _{\varepsilon \rightarrow 0} \varepsilon^{d} n_{F}(\mathcal{A}_\Psi, \varepsilon)$$
and $n_{F}(\mathcal{A}_\Psi, \varepsilon)$ is the minimum number of balls of radius $\leqslant \varepsilon$ which is necessary to cover $\mathcal{A}_\Psi$.

Take a covering of   $\mathcal{A}_\Psi$ by balls of radii less than $\varepsilon$ :
$$
\mathcal{A}_\Psi \subset \bigcup_{i=1} B\left(u_{i}, r_{i}\right), r_{i} \leq \varepsilon, u_{i} \in H
$$
where $B\left(u_{i}, r_{i}\right)$ denotes the ball in $H$ of center $u_{i}$ and radius $r_{i}$. Let $\theta=\theta_1$ and define
\begin{equation}\label{4.50}
\Psi(\omega)\phi=\Psi(1,\omega,\phi)
\end{equation}
for any $\phi\in H$ and P-a.e. $\omega\in \Omega$.
Then, it follows from the invariance of $\mathcal{A}_\Psi$ that
$$
\mathcal{A}_\Psi(\theta \omega) \subset \bigcup_{i=1} \Psi(\omega) B\left(u_{i}, r_{i}\right).
$$
In order to approximate $\Psi(\omega)$ by a linear map, we impose  the following almost surely uniformly differentiable assumption of $\Psi(\omega)$ on the attractor $\mathcal{A}_\Psi$.

$\mathbf{Hypothesis \  A3}$  The mapping $\Psi(\omega)$ is $\mathbb{P}$ almost surely differentiable, that is, $\mathbb{P}$ almost surely, for every $u$ in $\mathcal{A}_\Psi$, there exists a continuous linear operator $D \Psi(\omega, u): H\rightarrow H$, such that if $u, u+h \in \mathcal{A}_\Psi$, then
$$
|\Psi(\omega)(u+h)-\Psi(\omega) u-D \Psi(\omega, u) \cdot h| \leq K(\omega)|h|^{1+\alpha},
$$
where $K(\omega)$ is a random variable such that
$$
K(\omega) \geq 1, \ \mbox{for all}\ \omega \in \Omega,
$$
$\mathbb{E}(\ln K)<\infty$ and $\alpha>0$ is a number satisfying
$$
\alpha>0.
$$

For the  bounded linear operator $D \Psi(\omega, u)$ on $H$ and $n \in \mathbb{N}$, we set
$$
\alpha_{n}(D \Psi(\omega, u))=\sup _{\substack{G \subset H \\ \mathrm{~dim}\leqslant n}}\inf _{\substack{\phi \in G\\ \|\phi\|=1}}|D \Psi(\omega, u) \phi|
$$
and
$$
\omega_{n}(D \Psi(\omega, u))=\alpha_{1}(D \Psi(\omega, u)) \ldots \alpha_{n}(D \Psi(\omega, u)),
$$
where $\alpha_{n}(D \Psi(\omega, u))$ are the square roots of the eigenvalues of $D \Psi(\omega, u)^* D \Psi(\omega, u)$ corresponding to orthogonal eigenvectors $e_{n}$, which are in decreasing order. We set
$$
\alpha_{\infty}(D \Psi(\omega, u))=\inf _{n} \alpha_{n}(D \Psi(\omega, u))
$$
and further make the following assumptions.

$\mathbf{Hypothesis \  A4}$ There exists an integrable random variable $\bar{\omega}_{d}$, such that $\mathbb{P}$ almost surely,
$$
\omega_{d}(D S(\omega, u)) \leq \bar{\omega}_{d}(\omega)
$$
for any $u\in \mathcal{A}_\Psi$ and
$$
\mathbb{E}\left(\ln \left(\bar{\omega}_{d}\right)\right)<0 .
$$

Under the above assumptions, we have the following results concerning the dimension estimation of random attractors $\mathcal{A}_\Psi$ for $\Psi$, of which the proof is given in \cite{DA98,LR}.
\begin{lem}\label{lem4.1}
Assume that $\mathbf{Hypothesis \  A3-A4}$ are satisfied. Then, $\mathbb{P}$-a.s.
$$
d_{\mathrm{H}}(\mathcal{A}_\Psi) \leqslant d
$$
and
$$
d_{\mathrm{F}}(\mathcal{A}_\Psi) \leqslant \gamma
$$
for any $\gamma$ such that
$$
\gamma>\frac{\mathbb{E}\left[\max _{1 \leqslant j \leqslant d}\left(d q_{j}-j q_{d}\right)\right]}{-\mathbb{E} q_{d}},
$$
where $q_{j}=\log \bar{\omega}_{j}$.
\end{lem}

In the following, we verify $\mathbf{Hypothesis \  A3-A4}$. We first give the following result, which is a key ingredient to prove the $\mathbb{P}$ almost surely uniformly differentiability results of $\Psi(\omega)$.
\begin{pro}\label{pro1}
If $f: \mathcal{L}\rightarrow H$ is twice continuously differentiable, then for each $\varpi \in \mathcal{A}_\Phi$ and $h \in H$, there exists a continuous function $U^{\varpi, h}(t,\omega):[0, \infty)\times\Omega \rightarrow H$ such that
\begin{equation}\label{4.1a}
U^{\varpi, h}(t,\omega)=e^{-\tilde{L}t}\tilde{S}(t)h+\int_{0}^{t} e^{-\tilde{L}(t-s)}\tilde{S}(t-s)\left\{0, D \tilde{f}\left(P_1\Phi(s,\omega,\varpi)\right) P_{1} U(s)\right\} d s, \quad t \geqslant 0.
\end{equation}
Moreover, if $h\in D(\tilde{A})$, then $U(t,\omega)$ is a strong solution of the following variational equation on $H$.
\begin{equation}\label{4.1}
\left\{\begin{array}{l}
\displaystyle \frac{dU(t,\omega)}{dt}=\displaystyle  \tilde{A}U(t,\omega)-\mu U(t,\omega)+\left\{0, D \tilde{f}\left(P_1\Phi(s,\omega,\varpi)\right) P_{1} U(t,\omega)\right\},\\
U(0,\omega)=h \in H,
\end{array}\right.
\end{equation}
where  operators $\tilde{A}$ and $\tilde{f}$ are defined by \eqref{Atilde} and \eqref{2.4a}, $\Phi(t,\omega,\varpi)$ is RDS defined by \eqref{2.5d} with initial condition $h$.
\end{pro}
\begin{proof}
Let $$L_{1}(\omega)=\sup_{\varsigma \in P_{1}\mathcal{A}_\Phi}|D \tilde{f}(\varsigma)|,$$ where $$ |D\tilde{ f}(\varsigma)|=\sup_{\|\eta\|_{ \mathcal{L}} \leqslant 1}\|D \tilde{f}(\varsigma)\eta\|_{\mathbb{X}}. $$
Since $\tilde{f}$ is $C^{1}$ and $P_{1}\mathcal{A}_\Psi$ is compact, then $L_{1}(\omega)<\infty$. Given any $h\in D(\tilde{A})$, define $F_{\chi}:H \rightarrow H$ by
$$
F_{\varpi}(h)=\left\{0, D \tilde{f}\left(P_1\Phi(t,\omega,\varpi)\right) P_{1}h\right\}, \quad t \geqslant 0, \quad h \in H.
$$
It follows from the invariance of $\mathcal{A}_\Phi$ under $\Phi$ and $\varpi\in \mathcal{A}_\Phi$ that $\Phi(s,\omega,\varpi)\in \mathcal{A}_\Phi$ and hence $P_1\Phi(t,\omega,\varpi)\in P_{1}\mathcal{A}_\Phi$ and  $\left|D f\left(P_1\Phi(t,\omega,\varpi)\right)\right| \leqslant L_{1}(\omega)<\infty$, for all $t \geqslant 0$. This implies that $F_{\varpi}(\cdot)$ is Lipschitz continuous on $H$. Therefore the conclusion follows from  Pazy \cite[Theorem 6.1.5]{PA}.
\end{proof}

Now, we establish the following almost surely uniform differentiability results of $\Psi(\omega)$ on the random attractor  $\mathcal{A}_\Psi(\omega)$.
\begin{thm}\label{thm4.1}
The mapping $\Psi(\omega)$ is $\mathbb{P}$ almost surely differentiable, that is, $\mathbb{P}$ almost surely, for every $u$ in $\mathcal{A}(\omega)$, there exist a continuous linear operator $D \Psi(\omega, u): H\rightarrow H$, such that if $u, u+h \in \mathcal{A}(\omega)$, then
$$
\|\Psi(\omega)(u+h)-\Psi(\omega) u-D \Psi(\omega, u) \cdot h\| \leq K(\omega)\|h\|^{1+\alpha}
$$
where $K(\omega)$ is a random variable such that
$$
K(w) \geq 1, w \in \Omega
$$
and $\alpha>0$ is a number.
\end{thm}
\begin{proof}
We first claim that for any constant $t>0$ and $\chi, \chi+h\in \mathcal{A}(\omega)$, there exists a constant $L>0$ such that
$$\left\|\Psi(t,\omega,\chi)-\Psi(t,\omega,\chi+h)\right\|_{X} \leqslant L(t)\|h\|.$$
By Theorem \ref{thm2.1} and the relationship $\Psi(t, \omega,\chi(\omega))=\Phi(t,\omega, \varpi(\omega))+(z(\theta_{t+\cdot}\omega), z(\theta_{t}\omega))$ with $\varpi(\omega)=\chi(\omega)-(z(\theta_{\cdot}\omega), z(\omega))$, we have
\begin{equation}\label{4.2}
\Psi(t,\omega, \chi)=e^{-\tilde{L}t}\tilde{S}(t)\chi+\int_{0}^{t} e^{-\tilde{L}(t-s)}\tilde{S}(t-s)F(s, \theta_{s}\omega, \Phi(s,\omega, \varpi)) d s+(z(\theta_{t+\cdot}\omega), z(\theta_{t}\omega)),
\end{equation}
\begin{equation}\label{4.3}
\Psi(t,\omega, \chi+h)=e^{-\tilde{L}t}\tilde{S}(t)(\chi+h)+\int_{0}^{t} e^{-\tilde{L}(t-s)}\tilde{S}(t-s)F(s, \theta_{s}\omega, \Phi(s,\omega, \varpi+h))) d s+(z(\theta_{t+\cdot}\omega), z(\theta_{t}\omega)),
\end{equation}
from which it follows that
\begin{equation}\label{4.4}
\Psi(t,\omega, \chi+h)-\Psi(t,\omega, \chi)=e^{-\tilde{L}t}\tilde{S}(t)h+\int_{0}^{t} e^{-\tilde{L}(t-s)}\tilde{S}(t-s)[F(s, \theta_{s}\omega, \Phi(s,\omega, \varpi+h))-F(s, \theta_{s}\omega, \Phi(s,\omega,\varpi))]ds.
\end{equation}
Since $\|\tilde{S}(t)\|\leqslant e^{\max \{0,1 / 2+\sigma\} t}, t \geqslant 0$, we have
\begin{equation}\label{4.6}
\begin{aligned}
\|\Psi(t,\omega, \chi+h)-\Psi(t,\omega, \chi)\|&\leq e^{\varrho t}\|h\|+L_f\int_{0}^{t} e^{\varrho (t-s)}\|P_1[\Phi(s,\omega, \varpi+h)-\Phi(s,\omega,\varpi)]\|ds\\
&= e^{\varrho t}\|h\|+L_f\int_{0}^{t} e^{\varrho (t-s)}\|P_1[\Psi(s,\omega, \chi+h)-\Psi(s,\omega, \chi)]\|ds\\
&\leq e^{\varrho t}\|h\|+L_f\int_{0}^{t} e^{\varrho (t-s)}\|\Psi(s,\omega, \chi+h)-\Psi(s,\omega, \chi)\|ds.
\end{aligned}
\end{equation}
Multiplying both sides of \eqref{4.6} by $e^{-\varrho t}$ and taking into account the Gr\"{o}nwall inequality, we obtain
\begin{equation}\label{4.5}
e^{-\varrho t}\|\Psi(t,\omega, \chi+h)-\Psi(t,\omega, \chi)\|\leq e^{L_f t} \|h\|,
\end{equation}
and hence
\begin{equation}\label{4.500}
\|\Psi(t,\omega, \chi+h)-\Psi(t,\omega, \chi)\|\leq e^{(L_f+\varrho) t} \|h\|.
\end{equation}
Therefore, the claim holds by taking $L(t)=e^{(L_f+\varrho) t}$.

Next we prove that,  for any $t >0$, there exist $K(\omega)\geq1$ and $\alpha>0$ such that, if $\chi, \chi+h\in \mathcal{A}(\omega)$, then
\begin{equation}\label{4.7}
\|\Psi(t,\omega, \chi+h)-\Psi(t,\omega, \chi)-U^{\chi+h, \chi}(t,\omega)\|\leq K(\omega)\|h\|^{1+\alpha}.
\end{equation}
Let
\begin{equation}\label{4.8}
L_{2}(\omega):=\sup _{\xi \in \overline{co}\mathcal{A}(w)}\left|D^{2} f(P_1 \xi)\right|,
\end{equation}
where $\overline{co}\mathcal{A}(\omega) $ represents the closed convex hull of $ \mathcal{A}(\omega) $.
Since $f$ is $C^{2}$ and $ \mathcal{A}(\omega)$ is compact, $L_{2}<\infty$. By Proposition \ref{pro1}, we have
$$
U^{\chi, h}(t,\omega)=e^{-\tilde{L}t}\tilde{S}(t)h+\int_{0}^{t} e^{-\tilde{L}(t-s)}\tilde{S}(t-s)\left\{0, D \tilde{f}\left(P_1\Phi(s,\omega,\varpi)\right) P_{1} U(s)\right\} d s, \quad t \geqslant 0.
$$
For notation simplicity, we denote $y(t,\omega)\triangleq  \Psi(t,\omega, \chi+h)-\Psi(t,\omega, \chi)=\Phi(s,\omega,\varpi+h)-\Phi(s,\omega,\varpi)$ and $w(t,\omega)\triangleq  \Psi(t,\omega, \chi+h)-\Psi(t,\omega, \chi)-U^{\varpi, h}(t,\omega) $. Then, it follows  from \eqref{4.2} and \eqref{4.3}  that
\begin{equation}\label{4.10}
\begin{aligned}
\|w(t,\omega)\|&=\|\int_{0}^{t} e^{-\tilde{L}(t-s)}\tilde{S}(t-s) \{0, \tilde{f}\left(P_1\Phi(s,\omega, \varpi+h)\right)\\
&-\tilde{f}\left(P_1\Phi(s,\omega,\varpi)\right)-D \tilde{f}\left(P_1\Phi(s,\omega,\varpi)\right) P_{1} U(s)\}d s\| \\
&\leq \int_{0}^{t} e^{\varrho (t-s)}\|\tilde{f}\left(P_1\Phi(s,\omega,\varpi+h)\right)-\tilde{f}\left(P_1\Phi(s,\omega,\varpi)\right)-D \tilde{f}\left(P_1\Phi(s,\omega,\varpi)\right) P_{1} U(s)\|_\mathbb{X}d s\\
&\leq \int_{0}^{t} e^{\varrho (t-s)}\int_{0}^{1}  |D\tilde{ f}\left(P_1(\Phi(s,\omega, \varpi)+\vartheta y(s,\omega))\right)-D \tilde{f}\left(P_1\Phi(s,\omega, \varpi)\right) |d \vartheta\|P_1y(s,\omega)\|_\mathcal{L} d s \\
&+ \int_{0}^{t} e^{\varrho (t-s)}\|D\tilde{ f}\left(P_1\Phi(s,\omega,\varpi)\right)P_1w(s,\omega)\|_\mathbb{X}  d s \\
&\leq\int_{0}^{t}  e^{\varrho (t-s)}\int _ { 0 } ^ { 1 } \int _ { 0 } ^ { 1 } |D ^ { 2 } \tilde{f }\left(P_1(\Phi(s,\omega,\varpi)+\lambda \vartheta y(s,\omega))\right)|  \lambda \vartheta d \lambda d \vartheta\|P_1y(s,\omega)\|_\mathcal{L }^2  d s \\
&+ \int_{0}^{t} e^{\varrho (t-s)}\|D\tilde{ f}\left(P_1\Phi(s,\omega,\varpi)\right)P_1w(s,\omega)\|_\mathbb{X}    d s\\
&\leq\int_{0}^{t}  e^{\varrho (t-s)}\int _ { 0 } ^ { 1 } \int _ { 0 } ^ { 1 } |D ^ { 2 } \tilde{f }\left(P_1(\Phi(s,\omega,\varpi)+\lambda \vartheta y(s,\omega))\right)| d \lambda d \vartheta\|y(s,\omega)\|^2 d s \\
&+ \int_{0}^{t} e^{\varrho (t-s)}|D\tilde{ f}\left(P_1\Phi(s,\omega,\varpi)\right)|\|w(s,\omega)\|  d s.\\
\end{aligned}
\end{equation}
Since $\chi, \chi+h \in  \mathcal{A}_\Psi(\omega)$, it follows from the invariance Corollary \ref{cor3.1} that $\varpi, \varpi+h \in  \mathcal{A}_\Phi(\omega)$. Therefore, the invariance  of  $ \mathcal{A}_\Phi(\omega)$ under $\Phi$ implies that $\Phi(t,\omega, \varpi), \Phi(t,\omega, \varpi+h) \in  \mathcal{A}(\omega)$ for all $t \geqslant 0$. Therefore, $ P_1\Phi(t,\omega, \varpi)+\lambda \vartheta y(s,\omega)) \in \overline{c o}\left(\mathcal{A}_\Phi(\omega)\right)$, for all $\vartheta, \lambda \in[0,1]$, where $\overline{co}\mathcal{A}_\Phi(\omega) $ represents the closed convex hull of $ \mathcal{A}_\Phi(\omega) $. Thus, it follows from \eqref{4.8} and \eqref{4.5} and the fact $f$ is $C^{2}$ that
\begin{equation}\label{4.11}
\begin{aligned}
\|w(t,\omega)\|& \leqslant  L_{2}(\omega)\int_{0}^{t} e^{[2(L_f+\varrho)+\varrho](t-s)}\|h\|^{2} d s+L_{1}(\omega)\int_{0}^{t} e^{\varrho(t-s)} \left\|w(s,\omega)\right\| d s.
\end{aligned}
\end{equation}
Multiply both sides of \eqref{4.11} by $e^{-\varrho t}$ gives
\begin{equation}\label{4.11g}
\begin{aligned}
e^{-\varrho t}\|w(t,\omega)\|& \leqslant  \frac{-L_{2}(\omega) e^{-\varrho t}}{2(L_f+\varrho)+\varrho}\|h\|^{2}+L_{1}(\omega)\int_{0}^{t} e^{-\varrho s} \left\|w(s,\omega)\right\| d s.
\end{aligned}
\end{equation}
which implies, by the Gronwall inequality, that
\begin{equation}\label{4.12}
\begin{aligned}
e^{-\varrho t}\|w(t,\omega)\| \leqslant \frac{-L_{2}(\omega) e^{-\varrho t}}{2(L_f+\varrho)+\varrho}\|h\|^{2}+\frac{L_{1}(\omega)e^{L_{1}(\omega) t}}{-(L_f+\varrho)}(e^{-(\varrho+L_1)t}-1)\|h\|^{2} .
\end{aligned}
\end{equation}
Therefore, we have
\begin{equation}\label{4.12d}
\begin{aligned}
\|w(t,\omega)\| \leqslant \frac{-L_{2}(\omega)}{2(L_f+\varrho)+\varrho}(1+\frac{L_{1}(\omega)(1-e^{(L_{1}(\omega)+\varrho) t}}{-(L_f+\varrho)})\|h\|^{2}.
\end{aligned}
\end{equation}

Take  $D\Psi(\omega)h\triangleq U^{\varpi, h}(1,\omega)$, then it follows from \eqref{4.1a} that $D\Psi(\omega)$ is linear and continuous. Moreover, we have
\begin{equation}\label{4.12b}
\begin{aligned}
\|\Psi(\omega)(\chi+h)-\Psi(\omega) \chi-D \Psi(\omega, \chi) \cdot h\|&=\|\Psi(1,\omega,\chi+h)-\Psi(1,\omega \chi)-U^{\chi, h}(1,\omega)\|\\
&\leq \frac{-L_{2}(\omega)}{2(L_f+\varrho)+\varrho}(1+\frac{L_{1}(\omega)(1-e^{(L_{1}(\omega)+\varrho)}}{-(L_f+\varrho)})\|h\|^{2},
\end{aligned}
\end{equation}
what implies that the statement of Theorem \ref{thm4.1} holds by taking $\alpha=1$ and $K(\omega)=\frac{-L_{2}(\omega)}{2(L_f+\varrho)+\varrho}(1+\frac{L_{1}(\omega)(1-e^{(L_{1}(\omega)+\varrho)}}{-(L_f+\varrho)})$.
\end{proof}

We can now prove the main results of this paper.
\begin{thm}\label{thm4.2}
Assume that $\mathbf{Hypothesis \  A1-A4}$ are satisfied and  $f: L^{2}([-r, 0] ; H)\rightarrow H$ is twice continuously differentiable. Then, \eqref{2.4} admits a random attractor $\mathcal{A}_\Psi(\omega)$ that satisfies for $\mathbb{P}$-a.s.
$$
d_{\mathrm{H}}(\mathcal{A}_\Psi(\omega)) \leqslant d
$$
and
$$
d_{\mathrm{F}}(\mathcal{A}_\Psi(\omega)) \leqslant \gamma,
$$
for any $\gamma$ such that
$$
\gamma>\frac{\mathbb{E}\left[\max _{1 \leqslant j \leqslant d}\left(d q_{j}-j q_{d}\right)\right]}{-\mathbb{E} q_{d}},
$$
where $q_{j}=\log \bar{\omega}_{j}$.
\end{thm}
\begin{proof}
The existence of random attractor $\mathcal{A}_\Psi(\omega)$ has been proved in Theorem \ref{thm2.1} under  \break
$\mathbf{Hypotheses \  A1-A2}$. We show in the sequel the finite dimensionality of $\mathcal{A}_\Psi(\omega)$.
Let $\chi \in \mathcal{A}_\Psi(\omega)$, $h_i \in D(\bar{A})$ and
$U_i(t,\omega)\triangleq U^{\varpi, h_i}(t,\omega):[0, \infty)\times\Omega\times H \rightarrow H, i=1, \ldots, m$ be defined by
\begin{equation}\label{4.13}
U_i(t,\omega)=e^{-\tilde{L}t}\tilde{S}(t)h_i+\int_{0}^{t} e^{-\tilde{L}(t-s)}\tilde{S}(t-s)\left\{0, D \tilde{f}\left(P_1\Phi(s,\omega,\varpi)\right) P_{1} U(s)\right\} d s, \quad t \geqslant 0.
\end{equation}
It follows from Proposition \ref{pro1} that $U_i(t,\omega)$ satisfies the following variational equation on $H$.
\begin{equation}\label{4.14}
\left\{\begin{array}{l}
\displaystyle \frac{dU_i(t,\omega)}{dt}=\displaystyle  \tilde{A}U_i(t,\omega)-\mu U_i(t,\omega)+\left\{0, D \tilde{f}\left(P_1\Phi(s,\omega,\varpi)\right) P_{1}U_i(t,\omega)\right\},\\
U_i(0,\omega)=h_i \in H,
\end{array}\right.
\end{equation}
By a similar argument to that for (2.40) in \cite{TR} Chapter V, we obtain
\begin{equation}\label{4.15}
\frac{1}{2} \frac{d}{d t}\left|U_{1}(t,\omega) \wedge \cdots \wedge U_{m}(t,\omega)\right|_{\wedge^{m} H}^{2}=\left|U_{1}(t,\omega) \wedge \cdots \wedge U_{m}(t,\omega)\right|_{\wedge^{m} H}^{2} \operatorname{Tr}\left(G(t) \circ Q_{m}(t)\right),
\end{equation}
where $|\cdot|_{\wedge^{m} H}$ represents the exterior product and
\begin{equation}\label{4.16}
Q_{m}(t)=Q_{m}\left(\chi, h_1, \ldots, h_{m}\right)
\end{equation}
is the orthogonal projection of $H$ onto the space spanned by $U_{1}(t,\omega), \ldots, U_{m}(t,\omega)$ and $G(t)=G(t,\omega): H \rightarrow H$ is defined by
\begin{equation}\label{4.17}
G(t,\omega)h=\tilde{A}h-\mu h(t,\omega)+\left\{0, D \tilde{f}\left(P_1\Phi(s,\omega,\varpi)\right) P_{1}h\right\}.
\end{equation}
Therefore
\begin{equation}\label{4.18}
\begin{aligned}
&\left|U_{1}(t,\omega) \wedge \cdots \wedge U_{m}(t,\omega)\right|_{\wedge^{m} H} \\
&=\left|U_{1}(0,\omega) \wedge \cdots \wedge U_{m}(0,\omega)\right|_{\wedge^{m} H} \exp \left(\int_{0}^{t} \operatorname{Tr}\left(G(s,\omega) \circ Q_{m}(s)\right) ds\right) \\
&=\left|h_1 \wedge \cdots \wedge h_m\right|_{\wedge^{m} H} \exp \left(\int_{0}^{t} \operatorname{Tr}\left(G(s,\omega) \circ Q_{m}(s)\right) d s\right).
\end{aligned}\
\end{equation}

Let
\begin{equation}\label{4.19}
\begin{aligned}
q_{m}(t,\omega)=\sup _{\chi \in \mathcal{A}_\Psi(\omega), h_{i}\in D(\bar{A}),\|h_{i}\|_{H}\leq 1}\frac{1}{t} \int_{0}^{t}\operatorname{Tr}\left(G(s,\omega) \circ Q_{m}(s)\right) ds.
\end{aligned}\
\end{equation}
and
\begin{equation}\label{4.20}
\begin{aligned}
q_{m}(\omega)=\limsup _{t \rightarrow \infty} q_{m}(t,\omega).
\end{aligned}\
\end{equation}
Then we have
\begin{equation}\label{4.21}
\begin{aligned}
&\left|U_{1}(t,\omega) \wedge \cdots \wedge U_{m}(t,\omega)\right|_{\wedge^{m} H}\leq \left|h_1 \wedge \cdots \wedge h_m\right|_{\wedge^{m} H} \exp \left\{t q_{m}(t,\omega)\right\}.
\end{aligned}\
\end{equation}

Since by proposition \ref{pro1},  the mild solution of the initial value problem \eqref{4.1} depends continuously on initial data and $D(\bar{A})$ is dense in $H$ holds for all $\left\{h_{i}\right\} \in H, i=1, \ldots, m$. Therefore, we have the following estimation:
$$
\bar{\omega}_{m}(t,\omega) \leqslant \exp \left\{t q_{m}(t,\omega)\right\}
$$
and
$$
\pi_{m} \leqslant e^{q_{m}}.
$$
Thus
$$
\mu_{1}+\cdots+\mu_{m}=\ln \pi_{m} \leqslant q_{m},
$$
indicating the results hold by Lemma \ref{lem4.1}.
\end{proof}

\section{Conclusions}
In this paper, we have estimated the topological dimensions  of random attractor for the stochastic delayed semilinear partial differential equation \eqref{1}. In order to overcome the difficulty caused by the lack of Hilbert geometry, we recast the equation into a Hilbert space. One naturally wonders, whether we can estimate the dimension of attractors for SPFDEs in their natural phase space, i.e. Banach spaces. This requires to establish  the general framework to estimate the dimension of attractors of RDS in Banach spaces, which will be studied in the near future. Moreover, there are also SPFDEs on infinite domains which can  model the spatial-temporal patterns for the mature population of age-structured species  under random perturbations. The existence of random attractors for a stochastic nonlocal delayed reaction-diffusion equation on a semi-infinite interval have been studied in \cite{HZT}. However, little attention has been paid to the estimation of topological dimensions  of random attractor for the equation therein, which also deserves much effort in the future.

\medskip

\noindent{\bf Acknowledgement.}
This work was jointly supported by China Postdoctoral Science Foundation (2019TQ0089), Hunan Provincial Natural Science Foundation of China (2020JJ5344), the Scientific Research Fund of Hunan Provincial Education Department (20B353), China Scholarship Council(202008430247). \\
The research of T. Caraballo has been partially supported by Spanish Ministerio de Ciencia e
Innovaci\'{o}n (MCI), Agencia Estatal de Investigaci\'{o}n (AEI), Fondo Europeo de
Desarrollo Regional (FEDER) under the project PID2021-122991NB-C21 and the Junta de Andaluc\'{i}a
and FEDER under the project P18-FR-4509.\\
This work was completed when Wenjie Hu was visiting the Universidad de Sevilla as a visiting scholar, and he would like to thank the staff in the Facultad de Matem\'{a}ticas  for their help and thank the university for its excellent facilities and support during his stay.

\medskip
\noindent{\bf Data availability.}
No data has been used in the development of the research in this paper.

\medskip
\noindent
{\bf Conflict of interest.}
The authors declare that there are not any conflict of interest.

\end{document}